\documentclass[a4paper,11pt]{elsarticle}

\makeatletter
\def\ps@pprintTitle{%
  \let\@oddhead\@empty
  \let\@evenhead\@empty
  \def\@oddfoot{\reset@font\hfil\thepage\hfil}
  \let\@evenfoot\@oddfoot
}
\makeatother
\usepackage{amsmath}
\usepackage{amsthm}
\usepackage{amssymb}
\usepackage{enumerate}
\usepackage{ifthen}
\usepackage{calc}
\usepackage{tikz}
\usetikzlibrary[snakes]

\def\F{\mathcal{F}}

\def\bd{\mathbf{d}}
\def\bbd{\mathbf{bd}}
\def\bdd{\mathbf{dd}}
\def\bdf{\bd^{\F}}
\def\mc{\mathbf{maxC}}
\def\a{d^+}
\def\b{d^-}
\newcommand\junk[1]{}

\usepackage{verbatim}

\numberwithin{equation}{section}
\newtheorem{theorem}{Theorem}[section]
\newtheorem{lemma}[theorem]{Lemma}
\newtheorem{observ}[theorem]{Observation}

\theoremstyle{definition}
\newtheorem{definition}[theorem]{Definition}
\newtheorem{remark}[theorem]{Remark}
\def\qed{{\hfill$\Box$}}
\def\G{\mathbb{G}}

\def\C{\mathcal{C}}
\newcommand{\proba}[1]{#1}

\newif\ifdeveloping
\ifdeveloping
    \usepackage[notref,notcite]{showkeys}
\fi

\newif\ifproba
\ifproba
    \renewcommand{\proba}[1]{\ }
\fi

\newcommand\old[1]{}
\def\B{\mathbb{B}}

\def\<{\left\langle}
\def\>{\right\rangle}

\begin{document}
\begin{frontmatter}
\title{Constructing, sampling and counting graphical realizations of restricted degree sequences \tnoteref{t1}}

\author[renyi]{P\'eter L. Erd\H os\fnref{elp}}
\author[elte]{S\'andor Z.  Kiss\fnref{kiss}}
\author[renyi]{Istv\'an Mikl\'os\fnref{miklos}}
\author[renyi]{Lajos Soukup\fnref{soukup}}
\address[renyi]{Alfr\'ed R{\'e}nyi Institute of Mathematics, Re\'altanoda u 13-15 Budapest,
        1053 Hungary\\
        {\tt email}: $<$elp,miklosi,soukup$>$@renyi.hu}
\address[elte]{Institute of Automatization and Computation, 1111   Budapest, L\'agym\'anyosi \'ut 11, Hungary and University of Technology and Economics, Department of Algebra, 1111 Budapest, Egry J\'ozsef utca 1.  {\tt email}: kisspest@cs.elte.hu}
\fntext[elp]{Partly supported  by  Hungarian NSF, under contract NK 78439}
\fntext[kiss]{Partly supported by Hungarian NSF, under contract K77476 and NK105645}
\fntext[miklos]{Partly supported by Hungarian NSF, under contract  PD84297}
\fntext[soukup]{Partly supported by Hungarian NSF, under contract  NK 83726}
\tnotetext[t1]{We  acknowledge financial support from grant \#FA9550-12-1-0405 from the U.S. Air Force Office of Scientific Research (AFOSR) and the Defense Advanced Research Projects Agency (DARPA).}
\begin{abstract}
With the current burst of network theory (especially in connection with social and biological networks) there is a renewed interest on realizations of given degree sequences. In this paper we propose an essentially new degree sequence problem: we want to find graphical realizations of a given degree sequence on labeled vertices, where certain would-be edges are {\em forbidden}. Then we want to sample uniformly and efficiently all these possible realizations. (This problem can be considered as a special case of Tutte's $f$-factor problem, however it has a favorable sampling speed.)

We solve this {\em restricted degree sequence}  (or RDS for short) problem completely if the forbidden edges form a bipartite graph, which consist of  the union of a (not necessarily maximal) 1-factor and a (possible empty)  star. Then we show how one can sample the space of all realizations of these RDSs uniformly and efficiently when the degree sequence describes a {\em half-regular} bipartite graph. Our result contains, as special cases, the well-known result of Kannan, Tetali and Vempala on sampling regular bipartite graphs and a recent result of Greenhill on sampling regular directed graphs (so it also provides new proofs of them).

The RDS problem descried above is self-reducible, therefore our {\em fully polynomial almost uniform sampler} (a.k.a. FPAUS) on the space of all realizations also provides a {\em fully polynomial randomized approximation scheme} (a.k.a. FPRAS) for approximate counting of all realizations.
\end{abstract}
\end{frontmatter}

\section{Introduction}\label{sec:intro}
\noindent
In the last fifteen years, network theory has been undergoing an exponential grow. One of its more important problems is to algorithmically  construct networks with predefined parameters or uniformly sampling these networks with these given parameters. For general background,  the interested reader can turn to the now-classic book of Newman, Barab\'asi and Watts (\cite{NBW06}) or to the more recent book of Newman (\cite{newman}).

In this paper we study networks (or graphs as we like to refer them) with given degree sequences. We propose a new degree sequence problem class called {\bf restricted degree sequence problem} and solve its first instance: we build procedure to decide quickly whether there exists a graph $G$ with a given degree sequence, where $G$ completely avoids a predefined set of forbidden edges, then develop a fast mixing Markov Chain approach (in line of Kannan, Tetali and Vempala's approach, see \cite{KTV97}) to sample almost uniformly all different realizations of the sequence.

We will focus on the technical details of this problem;  we do not intend to survey the history of the degree sequence problems in general or their connections to other developments in network theory. One can find detailed information about this specific background  in  the recent paper of Greenhill (\cite{green}), or in our previous paper \cite{MES} (for which the current paper is a direct continuation). Therefore we will touch only the directly affected definitions and earlier results.

In the next section we discuss some of the known degree sequence problems and introduce our proposed new problem class. In Section~\ref{sec:factor} we will study in details one instance of this new problem class, describing and proving a fast Havel-Hakimi type greedy algorithm to construct realizations. Then in Section \ref{sec:MCMC} we discuss some known MCMC approaches to sample  different kind of degree sequence realizations and state our result about the rapid uniform sampling of this instance of our proposed model. In Sections \ref{sec:miles} and \ref{sec:anal}  we will describe our approach in details. In Section \ref{sec:counting} we show that the studied problem is a self-reducible one, therefore our almost uniform sampling method provides good approximate counting the set of all realizations. Finally in the Appendix we will discuss a necessary technical detail of the required, but very slight generalization of the "Simplified Sinclair's method" introduced in \cite{MES}.

\section{Degree sequences and restricted degree sequences}\label{sec:degree}
\noindent
Let's fix a labeled underlying vertex set $V$ of $n$ elements. The {\em degree sequence} $\bd(G)$ of a {\em simple} graph $G=(V,E)$ is the  sequence of its vertex degrees: $\bd(G)_i= d(v_i).$ (Here "simple" means that there are no loops or multiple edges. In general, in cases where multiple edges and/or loops are not excluded the corresponding results are easier. See for example Ryser, \cite{R57}.)  A non-negative integer sequence $\bd= (d_1,\ldots, d_n)$ is {\em graphical} iff $\bd(G) = \mathbf{d}$ for some simple graph $G$, and then $G$ is a {\em graphical realization} of $\mathbf{d}$. The simplest method (a straightforward greedy algorithm) to decide the graphicality of an integer sequence was discovered by Havel  (\cite{H55}), and was rediscovered independently later by Hakimi (\cite{H62}). The validity of their algorithm is proved by the means of the  {\em swap} operation. It is defined as follows:

Let $G$ be a simple graph and assume that $a,b,c$ and $d$ are different vertices.  Furthermore, assume that $(a,c), (b,d) \in E(G)$ while $(b,c), (a,d) \not \in E(G)$. Then
\begin{equation}\label{eq:swap}
E(G')= E(G) \setminus \{(a,c), (b,d)\} \cup \{(b,c), (a,d)\}
\end{equation}
is another realization of the same degree sequence. We call this operation a {\bf swap} and will denote by $ac, bd \Rightarrow bc, ad.$

\medskip\noindent
The analogous notions for bipartite graphs are the following: if $B$ is a simple bipartite graph then its vertex classes will be denoted by $U(B)=\{ u_1,\ldots, u_k\}$ and $ W(B)=\{w_1,\ldots,w_\ell\}$, and we keep the notation $V(B)=U(B)\cup W(B)$.  The {\em
bipartite degree sequence} of $B$, $\bbd(B)$ is defined as follows:
$$
\bbd(B)=\Big(\big (d(u_1), \ldots,d(u_k)\bigr ), \bigl (d(w_1),\ldots,d(w_\ell)\bigr )\Big ).
$$
We can define the swap operation for bipartite realizations similarly to (\ref{eq:swap}) but we must take some care: it is not enough to assume that $(b,c), (a,d) \not \in E(G)$ but we have to know that $a$ and $b$ are in the same vertex class, and $c$ and $d$ are in the other.

To make clear whether a vertex pair can form an edge in a realization or not we will call a vertex pair a  {\bf chord} if it can hold an actual edge in a realization. Those pairs which cannot accommodate an edge are the {\bf non-chords}. (Pairs from the same vertex class of a bipartite graph are non-chords.)

\medskip\noindent
For directed graphs we consider the following definitions: Let $\vec G$ denote
a directed graph (no parallel edges, no loops, but oppositely directed edges between two vertices are allowed) with vertex set $X(\vec G) = \{x_1,x_2,  \ldots,x_n \}$ and edge set $E( \vec G)$. We use the bi-sequence
$$
\bdd(\vec G) =\Big (\left (\a_1,\a_2,\ldots,\a_n \right ), \left (\b_1, \b_2,
  \ldots, \b_n\right ) \Big )
$$
to denote the degree sequence, where $\a_i$ denotes the out-degree of vertex $x_i$ while $\b_i$ denotes its in-degree. A bi-sequence of non-negative integers is called a {\em directed degree sequence} if there exists a  directed graph $\vec G$ such that $\mathbf{(\a,\b)}=\bdd(\vec G)$. In this case we say that $\vec G$ {\em  realizes} our directed degree sequence.

We will apply the following {\em representation} of  the directed graph $\vec G:$  let $B({\vec G})=(U,W; E)$ be  a bipartite graph where each class consists of one copy of every vertex.  The edges adjacent to a vertex $u_x$ in class $U$ represent the out-edges from $x$, while the edges adjacent to a vertex $w_x$ in class $W$ represent the in-edges to $x$ (so a directed edge $xy$ is identified with the edge $u_xw_y$). If a vertex has no in- (out-) degree in the directed version, then we delete the corresponding vertex from $B({\vec G}).$ (This representation is an old trick, applied already by Gale \cite{G57}.)  There is no loop in our directed graph, therefore there is no $(u_x,v_x)$ type edge in its bipartite realization - these vertex pairs are non-chords.

\bigskip\noindent
In this paper we will study the following common generalization of all the previously mentioned degree sequence problems:  \\
The {\bf restricted degree sequence} problem $\bd ^{\F}$ consists of a degree sequence $\bd$ and a set $\F \subset \binom{V} { 2}$ of {\bf forbidden} edges.  The problem, as it is in the original problem, is to decide whether there is a simple graph $G$ on $V,$ completely avoiding the elements of $\F,$ which provides the given degree sequence, furthermore to design a uniform sampler of all realizations.

The {\bf bipartite restricted degree sequence} problem $\bbd^{\F}$ consists of a
bipartite degree
sequence $\bbd$ on $(U,W)$, and a set $\F \subset [U,W]$ of {\bf forbidden} edges.
The problem, as it is in the original problem, is to decide whether there is a
bipartite  graph $G$ on $(U,W)$ completely avoiding the elements of $\F,$ which
provides the given degree sequence.

Clearly a bipartite restricted degree sequence problem $\bbd^{\F}$ on $(U,W)$
is  a restricted degree sequence problem $\bd^{\F'}$ on $U\cup W$,
where $\F'=\F\cup [U]^2\cup[W]^2$.

It is quite obvious that the restricted degree sequence problem is a special case of Tutte's $f$-factor problem (\cite{T52}). (This was essentially pointed out already in the seminal paper of Paul Erd\H{o}s and Tibor Gallai,  see  \cite{EG60}). However, while Tutte's approach provides a polynomial time algorithm to decide whether a given $\bd ^{\F}$  restricted degree sequence problem can be satisfied, it does not provide an easy greedy algorithm to do so and also not really suitable to generate all possible realizations. It is important to add that while the fundamental result of Jerrum, Sinclair and Vigoda on sampling perfect matchings in graphs (\cite{JSV}) provides a uniform sampling approach for the possible realizations, their method is not useful in  practice. We will return to this issue at the end of Section \ref{sec:counting}.

As we already mentioned the RDS notion is a common generalization of several "classical" degree sequence problems. For instance, the case of bipartite degree sequences is such an example. Another  well known example is the case of directed degree sequences:  in its bipartite representation,  a forbidden 1-factor excludes the directed loops in the original directed graph.

\begin{definition}\label{def:cc}
Let $\bd^{\F}$ be a restricted degree sequence problem and let $G$ be a realization of it.
The sequence of  vertices $\C=(x_1, x_2,\ldots, x_{2i})$ is a {\bf chord-circuit} if:
({\bf D1})  all pair $x_1x_2, x_2x_3, \ldots x_{2i-1}x_{2i}, x_{2i} x_1$ are chords; and ({\bf D2}) each of these chords is different. \\
A chord-circuit is {\bf elementary} if ({\bf D3}) no vertex occurs more than twice;  furthermore  ({\bf D4}) when two copies of the same vertex exist, then their distance along the circuit is odd.   \\
The chord-circuit  $\C$  {\bf alternate}s  if   the chords along $\C$  are edges and non-edges in turn (for example  $x_{2j-1} x_{2j}$ are edges for $1\le j \le i$, while the other chords are not edges in $G$). \\
Deleting the actual edges along $\C$ from $G$ and adding the other chords as edges constructs a new graph  $G'$ which is again a realization of $\bd^\F.$ This operation is known as a {\bf circular $C_{2i}$-swap} and denoted by $S_{\C}$. \\
Finally, two {\em different} vertices of the alternating chord-circuit $\C$ form a {\bf potential vertex pair} (or {\bf PV-pair} for short) if  (i) this pair is not a chord along the circuit; and (ii) the distance of the vertices  along the sequence (which is the number of chords between them) is odd and it is not 1.  If each PV-pair  is a non-chord (that is $\in \F$), then this circular swap is called a {\bf  $\F$-compatible swap} or {\bf $\F$-swap} for short.
\end{definition}
The $\F$-swap is one of the central notions of this paper. When $i=2$ then the circular $C_4$-swap coincides with the classical Havel--Hakimi swap. When $i=3$ then we get back the notion of the {\bf triangular $C_6$-swap}, which was introduced in paper \cite{EKM} in connection with directed degree sequences.

We define the {\bf weight} of the $\F$-compatible circular $C_{2i}$-swap as $w(C_{2i})=i-1.$ This definition is in accordance with the definitions of the weight of the classical HH-swaps, and the weight of a triangular $C_6$-swap in paper \cite{EKM}. Furthermore it is well known (see for example  again \cite{EKM}) that in case of simple graphs $(i-1)$ Havel-Hakimi swaps are needed to alternate the edges along $C_{2i}.$ As we will see next the same applies for any (elementary) circular $C_{2i}$-swap:
\begin{lemma}\label{th:weight}
Let $G$ be a realization of $\bd^\F$ and let the elementary chord-circuit $\C$ of length $2i$ alternate. Then the  circular $\C$-swap operation can be carried out by a sequence of $\F$-swaps of  total weight $i-1.$
\end{lemma}
\noindent More precisely  there exists a sequence  $G=G_0, G_1,\ldots, G_\ell$ of realizations such that for each $j=0,\ldots, \ell-1$ there exists an $\F$-swap operation from $G_j$ to $G_{j+1},$  the difference between $G$ and $G_\ell$ is exactly the alternating circuit $\C,$ finally the total weights of those $\F$-swap operations is $i-1.$
\begin{proof}
We apply mathematical induction for the length of the chord-circuit: assume this is true for all circuits of length at most $2i-2$. (For $i=3$ this is the classical Havel-Hakimi swap.) Then take an alternating elementary chord-circuit $\C$ of length $2i$ in a realization of $\bd ^\F.$

If each PV-pair in $\C$ is a non-chord, then the circular $C_{2i}$-swap itself is a $\F$-swap of weight $i-1$. So we may assume that there is a PV-pair $uv$ in $\C$ which is a chord. This chord together with the two "half-circuits" of $\C$ form chord-circuits $\C_1$ and $\C_2$ using the chords of the original circuit $\C$ and twice the chord $uv$. One of them, say $\C_1,$ is alternating. The length of $\C_1$ is $2j < 2i$ therefore there exists a $\F$-compatible swap sequence of total weight $j-1$ to process it. After the procedure the status of $uv$ will alter into the other status. With this new status of the chord the circuit $\C_2$ becomes an alternating one, so it can be processed with $\frac{2i + 2 - 2j}{2} - 1$ total weight - and after this procedure the chord $uv$ is switched back to its original status. So we found a swap sequence of total weight $i-1$ which finishes the proof.
\end{proof}

\bigskip\noindent {\bf The space of all realizations}: Consider now the set of all possible realizations of a restricted graphical degree sequence $\bd^\F.$ Let $G$ and $H$ be two different realizations. The natural question is whether $G$ can be transformed into $H$ using $\F$-swaps?

For classical degree sequences this problem was  solved affirmatively already in 1891 by Petersen in the (by now  almost completely forgotten) paper \cite{pet}. Havel's paper \cite{H55} also provided (an implicit)  solution. For bipartite graphs (with possible multiple edges but no loops) this was done by Ryser (\cite{R57}). For simple bipartite graphs it was folklore. Finally for directed graphs it was done in \cite{KW73} (and later rediscovered in \cite{EMT09}).
\begin{theorem}\label{th:general-irreducible}
The space of all realizations $\G=(\mathbb{V},\mathbb{E})$  of the restricted degree sequences problem $\bd^\F$ is connected. Therefore the usual Markov chain defined on $\G$ is irreducible.
\end{theorem}
\begin{proof}
What we have to prove is the following: let $G$ and $H $ be two realizations of $\bd^\F.$ Then we have to find a series of realizations $G=G_0, \dots, G_{i-1},$ $G_i=H,$ such that for each $j=0,\ldots, i-1$ there exists an $\F$-swap from $G_j$ to $G_{j+1}.$
(Such realization pairs form the edge set of the Markov chain $\G.$)

Consider the symmetric difference of the edge sets of the two realizations: $\Delta = E(G) \triangle E(H).$ This set is  two-colored  by the original hosts of the edges: there are $G$-edges and $H$-edges. It is clear that in the graph $\mathcal{G}=(V, \Delta)$ for each vertex $v$ the numbers of $G$-edges and $H$-edges incident to $v$ are the same: $d_G(v) = d_H(v).$ By Euler's method it can be decomposed into alternating circuits $\C_1,\ldots, \C_\ell.$

Let's recall that a {\bf circuit}  in a simple graph $G$ is  a sequence of vertices $v_0,\ldots,v_{2t}$, where $v_0=v_{2t}$  s.t. the consecutive vertices are adjacent  and each edge can be used at most once. Note that there can also be other indices $i<j$ such that $v_i=v_j$. A circuit is called a {\bf cycle}, if it is simple, i.e., for any $i<j$, $\;v_i=v_j$ only if $i=0$ and $j=2t$. A circuit is {\bf alternating} for $G$ and $H$ if the edges come in turns from $E(G)$ and $E(H).$ When this is the case then the corresponding chord-circuit in realization $G$ (as well in $H$) is alternating.

We can find a decomposition  that no circuit contains a vertex $v$ twice s.t. their distance $\delta$ (the number of edges between them) is even. Indeed, since $\mathcal{G}$ is simple, therefore $\delta$ is at least four, consequently the vertex $v$ can split the original circuit into two smaller, but still alternating circuits. So for example an alternating circuit decomposition with maximal number of circuits has this property. It also implies that no circuit may contain a
vertex three times otherwise at least two copies of the vertex would be of even distance from each other.

The application of Lemma \ref{th:weight} proves that each circuit $\C$ can be processed with $|\C|/2 -1$ total weight. This finishes the proof.
\end{proof}

In the paper \cite{EKM} the following formula was developed for the required minimum weight of transforming one realization in to an other one of the same (unrestricted) degree sequence: Consider the realizations $G$ and $H$ and the symmetric difference $E \Delta F$ of their edge sets.  Denote $ \mc (G,H)$ the maximum possible number of the alternating circuits in such a decomposition.
\begin{theorem}[Erd\H{o}s - Kir\'aly - Mikl\'os \cite{EKM}]\label{th:SD-main1}
For any pair $G,H$ of  realizations  of a given degree sequence the weight of the shortest swap sequence between the two realizations is
\begin{equation}\label{eq:dist1}
\mathbf{dist} (G,H) = \frac{|E(G) \Delta E(H)|} {2}  - \mc(G,H).
\end{equation}
\end{theorem}
\noindent Now we show that the analogous result holds for any restricted degree sequence.
\begin{theorem}\label{th:SD-general}
Let $G, H$ be two realizations of the same restricted degree sequence $\bd^\F.$ Then
\begin{equation}\label{eq:dist2}
\mathbf{dist}_{\F} (G,H) = \frac{|E(G) \Delta E(H)|} {2}  - \mc(G,H).
\end{equation}
\end{theorem}
\begin{proof}
We show that
\begin{equation}
 RHS=^1\mathbf{dist} (G,H)\le^2 \mathbf{dist}_{\F} (G,H)\le^3 RHS.
\end{equation}
The first equality is just Theorem \ref{th:SD-main1}. To check the second inequality assume that the $\mathcal F$-swap sequence
\begin{equation}\label{eq:fwasp_g_to_H}
\mathcal S_{\C_1},\dots, \mathcal S_{\C_n}
\end{equation}
transfers $G$ to $H$. If the length of $\C_j$ is $2k_j$, then, applying Theorem \ref{th:SD-main1} again, we have that $\mathcal S_{\C_j}$ can be obtained as a composition of $k_j-1$ many ``standard''  swaps $\mathcal S_{\C_{j,1}},\dots,  \mathcal S_{\C_{j,k_j-1}}$. So the ``standard'' swap sequence
\begin{equation}
\mathcal S_{\C_1,1},\dots, \mathcal S_{\C_{1, j_1-1}},
\mathcal S_{\C_2,1},\dots, \mathcal S_{\C_{2, j_2-1}}\dots
\mathcal S_{\C_n,1},\dots, \mathcal S_{\C_{n, j_n-1}}
\end{equation}
transfers $G$ to $H$, and the weight of this sequence is $\sum_{j=1}^{n}(k_j-1)$,
which is exactly the weight of the $\mathcal F$-swap sequence
$\mathcal S_{\C_1},\dots, \mathcal S_{\C_n}$ from (\ref{eq:fwasp_g_to_H}).
Thus the second inequality holds.

Finally the third inequality holds by   Lemma \ref{th:weight}.
\end{proof}

\section{The star+factor restricted degree sequences}\label{sec:factor}
\noindent
We turn our interest now for the following  specialized restricted degree sequence problem: $\bd^\F$ is called a {\bf  Star -- 1-Factor Restricted Degree Sequence} problem (or {\bf star+factor} problem for short), if
\begin{enumerate}
\item[($\Psi$)] the set $\F$ of the forbidden edges is a bipartite graph where the edges are the union of an 1-factor and a star with center $s$.
\end{enumerate}
Similarly, if $\bbd$ is a bipartite degree sequence, and ($\Psi$) holds for $\F$,
then $\bbd^\F$ is called a {\bf  Bipartite Star -- 1-Factor Restricted Degree
Sequence} problem (or {\bf bipartite star+factor} problem for short).

\medskip\noindent As we already mentioned, Tutte's $f$-factor theorem can always be utilized to find actual graphical realizations of our star+factor restricted degree sequence. However in this special case we can apply a Havel-Hakimi type greedy algorithm to construct such  realizations.

Consider our $\bd^\F$ star+factor  degree sequence problem. For a given vertex
$x \in V$ denote $C(x)$ the set of those vertices in $V$ which form chords
together with $x.$ If
\begin{enumerate}[{\rm (i)}]
\item for each $y \in C(x)$ has at most one vertex, denoted by $y^{\F}$, such that pair $y y^{\F}$ belongs to $\F$ (it is a non-chord), furthermore if $y^\F = z^\F$ then $y=z$,
\end{enumerate}
then we say that $C(x)$ is {\bf normal}, and we fix an order $\prec_x$ on $C(x)$
such that
\begin{enumerate}[{\rm (i)}]
\addtocounter{enumi}{1}
\item   if  $d(y) > d(z)$, or $d(y) = d(z)$ and  $d(y^\F) > d(z^\F),$ then  $y \prec_x z$.
\end{enumerate}
\begin{lemma}\label{th:HH}
Let $G$ be a graphical realization of our star+factor RDS $\bdf,$ let $x\in V$ and assume that $C(x)$ is normal with the order $\prec_x$. Assume furthermore that $y\prec_x z$ and $xz \in E$ while $xy \not\in E.$  Then there exists an alternating chord-circuit $\C=(y,x,z,v_1,\ldots,v_i)$ with $i=1$ or $3$ such that if we carry out $S_C$, then in the acquired new realization we have $\Gamma_{G'}(x)=\Gamma_G(x)\setminus \{z\} \cup \{y\}.$
\end{lemma}
\noindent This statement is actually almost the same as Lemma 4 in \cite{EMT09} and its proof could be recalled. However here we give a complete proof. On one hand this keeps this paper self-contained, on the other hand  paper  \cite{EMT09} uses a different language.
\begin{proof}
We have $xz \in  E$ but $xy \not\in E$. At first assume that there exists a vertex $u \ne x,y,z$, such that $uy \in E,$ and $uz \not\in E$ but $u \ne z^\F.$ When such vertex exists then $\C=(x,z,u,y)$ is a suitable alternating chord-circuit as the $xz, uy \Rightarrow xy, uz$ swap shows.

\begin{definition}
From now on the notation $xz, uy \Rightarrow xy, uz$ always means that all pairs are chord, $xz, uy$ are edges, $xy, uz$ are non-edges and we consider the new realization created bye the indicated swap.
\end{definition}

We continue the proof: when $d(y) > d(z)$ then this happens automatically since $y$ belongs to at most one forbidden pair. However, if $d(y)=d(z)$ then it can happen that $z^\F y \in E$ and
\begin{equation}\label{eq:similar}
\forall u\ne x, y,z, y^\F, z^\F\quad \hbox{ we have }\quad  uy \in E \Leftrightarrow uz \in E.
\end{equation}
 It is important to observe that in this case $y^\F z \not\in E,$ otherwise some $u$ would not satisfy (\ref{eq:similar}) (in order to keep $d(y)=d(z)$).

\medskip\noindent
So the only case when we do not find automatically an appropriate swap with $x$, $y$ and $z$ is when $d(y)=d(z)$, $yz^\F$ is an edge and $zy^\F$ is a chord but not an edge. In this case, we can find a $ u\ne y,z$ such that $y^\F u \in E$ but $z^\F u \not\in  E$ since $d(y^\F) \ge d(z^\F)$.

Now $\mathcal C=(y,x,z,y^\F, u, z^\F,y)$ is the required alternating chord circle.  See the figure below. The three line types denotes the edges, the chords which are non-edges, finally the forbidden non-chords.
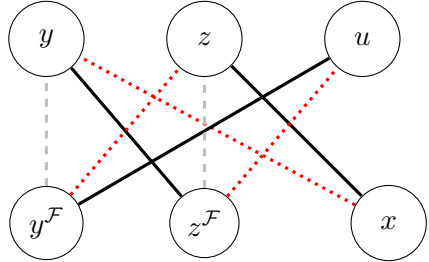
\begin{figure}[h!]
\center{
\begin{tikzpicture}[scale=0.35]
\node at (-1,8) [shape=circle,minimum size=10mm,draw] (p1) {$y$};
\node at (5,8) [shape=circle,minimum size=10mm,draw] (p2) {$z$};
\node at (11,8) [shape=circle,minimum size=10mm,draw] (p3) {$u$};
\node at (-1,1) [shape=circle,draw] (p4) {$y^\F$};
\node at (5,1) [shape=circle,draw] (p5) {$z^\F$};
\node at (12,1) [shape=circle,minimum size=10mm,draw] (p6) {$x$};
\draw [very thick] (p2) -- (p6);
\draw [very thick, dotted, red] (p6) -- (p1);
\draw [very thick] (p4) -- (p3);
\draw [very thick] (p5) -- (p1);
\draw [very thick,dotted,red] (p5) -- (p3);
\draw [very thick,dotted,red] (p4) -- (p2);
\draw [very thick,dashed,gray!50] (p4) -- (p1);
\draw [very thick,dashed,gray!50] (p5) -- (p2);
\end{tikzpicture}
 }
\end{figure}
\end{proof}
\noindent
Lemma \ref{th:HH} provides the following easy Havel-Hakimi type greedy algorithm to construct at least one graphical realization of our restricted degree sequence problem.

\medskip\noindent {\bf HH-algorithm} for a star+factor degree sequence problem. (Recall that the center of the forbidden star is denoted with $s$. Since a star can be empty, we can assume that $s$ is always defined):
\begin{enumerate}[ ($H_1$)]
\item take an  ordering $\prec_s$ on $C(s)$ (which is normal) and connect $s$ to the first $d(s)$ vertices (with respect to $\prec_s$) of $C(s).$ Delete $s$ and update the degrees of the used vertices accordingly.
\item take the remaining vertices one by one and repeat the process.
\end{enumerate}

\begin{theorem}[Generalized HH-theorem for the star+factor RDS problem] \label{th:general-HH}
The $\bd ^\F$  star+factor restricted degree sequence is graphical if and only if the previous greedy HH-algorithm provides a realization.
\end{theorem}
\begin{proof}
Similarly to the proof of the original HH-theorem, the recursive application of Lemma \ref{th:HH} proves the statement. And we can apply it recursively indeed: We start our construction with the vertex $s$ therefore $C(s)$ satisfies condition (i) before Lemma \ref{th:HH} (it is normal). Furthermore when $s$ is deleted, then  the remaining forbidden edges are from the original 1-factor, so the lemma applies automatically at each subsequent step.
\end{proof}
\medskip\noindent In the case of a bipartite degree sequence $\bd$ the situation is very similar: we define the normality of any $C(x)$ formally the same way as before. The definition of the order $\prec_x$ is also the same as before.

Here it is interesting to observe, that if $x$ is in class $U$ than $\C(x)$ is subset of class $W,$ the vertices $y^\F$ and $z^\F$ belong again to class $U$, so $u\in W$  (and, finally,  the forbidden edges define those vertices are elements of $\F$). Furthermore, as in any bipartite degree sequence problem, if we determine all edges adjacent to the vertices in $U$ then we automatically placed all edges adjacent to the vertices in $W$ as well.
\begin{theorem}\label{th:general-HHB}
Lemma \ref{th:HH} and Theorem \ref{th:general-HH} apply for the vertices of class $U$ in case of bipartite star+factor restricted degree sequences without any changes.
\end{theorem}
\begin{proof}
Indeed, by definition, the center $s$ of the forbidden star belongs to $U$. Also by definition, in the proof of the Lemma  the vertices we are considering are from the vertex class $U$ only: $x\in U.$ Therefore $C(x)$ is a subset of $W.$ Consequently the vertices $y^\F$ and $z^\F$ belong to vertex class $U$ therefore the vertex $u$ must belong to vertex class $W$ again. Therefore all forbidden edges considered in the proofs belong to $\F$. So the proofs apply without changes for the bipartite case as well.
\end{proof}

\section{Sampling uniformly half-regular, bipartite star+factor  res\-tricted degree sequences}\label{sec:MCMC}

\noindent
In this section we describe the main result of our paper which is an MCMC algorithm for  (almost) uniform sampling of the space of all realizations of the

\medskip\noindent{\bf Half-regular, Bipartite Star -- 1-Factor Restricted Degree Sequence} problem: Let $\bd$ be a bipartite degree sequence with a star+ 1-factor type forbidden edge set  $\F$, where the center of the star is denoted by $s$ (and belong to $U).$ Furthermore let
\begin{enumerate}
\item[($\Phi$)] the degree sequence $\bd$ is a {\bf half-regular} bipartite one: it requires a $B=(U,W; E)$ bipartite graph  where each vertex $u\in U$ -- with the one possible exception $s$ -- has the same degree. We will write $V$ for $U \cup W.$
\end{enumerate}
We will apply a  Markov Chain Monte Carlo method for almost uniform sampling of all possible realizations of our $\bd^\F.$ Originally the MCMC method for realizations' sampling was  proposed by Kannan, Tetali and Vempala (1999, \cite{KTV97}).  They conjectured that their process is {\em rapidly mixing} on the realizations of any (unrestricted) degree sequences, i.e., starting from an arbitrary realization of the degree sequence, the process reaches a completely random realization in reasonable (i.e., polynomial) time. They managed to prove the result for bipartite regular graphs. Their conjecture was proved for arbitrary regular graphs by Cooper, Dyer and Greenhill (2007, \cite{CDG07}). (Their result does not automatically generalize the previous result, since their version does not allow forbidden edges.) An analogous theorem was proved by Greenhill on regular directed graphs (\cite{green}).  Mikl\'os, Erd\H{o}s and Soukup proved  in \cite{MES} that this Markov process is also rapidly mixing on each bipartite {\bf half-regular} degree sequence (here there is no exceptional vertex $s)$. In this paper we will prove that the analogous Markov process is rapidly mixing on the half-regular star+factor bipartite restricted  degree sequence problem.

The state space of our {\bf Markov chain} is the graph $\G=(V( \G), E(\G))$ where $V(\G)$ consists of all possible realizations of our  problem, while the edges represent the possible swap operations: two realizations (which will be indicated by upper case letters like $X$ or $Y$) are connected if there is a valid $\F$-swap operation which transforms one realization into the other one (and  the inverse swap transforms the second one into the first one as well).

The {\em transition (probability) matrix}  $P$ is defined as follows: let the current realization be $G$. Then
\begin{enumerate}[(a)]
\item with probability $1/2$ we stay in the current state (that is, our Markov chain is {\bf lazy});
\item with probability $1/4$ we choose uniformly two-two vertices $u_1,u_2;v_1,v_2$ from classes $U$ and $W$ respectively and perform the swap if it is possible;
\item finally with probability $1/4$  choose three - three vertices  from $U$ and $W$ and check whether they form three pairs of forbidden chords. If this is the case then we perform a circular $C_6$-swap if it is possible.
\end{enumerate}
Here cases (b) and (c) correspond to Lemma \ref{th:HH}. The swaps moving from $G$ to its image $G'$ is unique, therefore the probability of this transformation (the {\em jumping probability} from $G$ to $G'\ne G$) is:
\begin{equation}\label{eq:prob}
\mathrm{Prob}(G \rightarrow_b   G'):= P(G' | G) = \frac{1}{4} \cdot \frac{1}{\binom{|U|}{2} \binom{|W|}{2}},
\end{equation}
and
\begin{equation}\label{eq:prob1}
\mathrm{Prob}(G\rightarrow_c G'):= P(G' | G) = \frac{1}{4} \cdot \frac{1}{\binom{|U|}{3} \binom{|W|}{3} }.
\end{equation}
(More precisely these are the probabilities that these vertex sets will be checked against making a swap.) The probability of transforming $G$ to $G'$ (or vice versa) is time-independent and {\em symmetric}. Therefore $P$ is a symmetric matrix, where  the entries in the main diagonal are non-zero, but (probably) different values. As we discussed it   earlier (see Theorem \ref{th:general-irreducible}), our Markov chain is  {\em irreducible} (the state space is connected), and it is clearly aperiodic, since it is lazy. Therefore, as it well known,  our Markov process is reversible with the uniform distribution as the globally stable stationary distribution.

Our main result is the following:
\begin{theorem}\label{th:main}
The Markov process defined above is  rapidly mixing on each bipartite half-regular degree sequence with a forbidden star and a forbidden $1$ -  factor.
\end{theorem}
\noindent This result supersedes the previously mentioned three results (\cite{KTV97, green, MES}) (except that this does not care of the actual mixing time, it just proves that it is polynomial). We want to add, however, that the  {\em friendly path} method, described in \cite{MES}, was not intended to handle half-regular bipartite degree sequences only but all bipartite degree sequences. Therefore we think that that method should not be completely neglected.

\bigskip\noindent There are several different methods to prove fast convergence of a Markov chain, here we apply a specialized version of Sinclair's seminal {\em multicommodity flow method} (\cite{S92}), the so called {\em simplified Sinclair's method}, developed in \cite{MES}\footnote{More precisely we need to slightly generalize it. It will be discussed in the Appendix.}:

\bigskip\noindent {\bf Simplified Sinclair's method}: We fix a half-regular star+factor restricted bipartite graphical sequence problem  $\bd^\F$. Our bipartite degree sequence is $\bd=\big (\mathbf{a}, \mathbf{b}\big )$ where the  vector $\mathbf{a}$ contains the degrees in class $U$ while $\mathbf{b}$  contains the degrees in class  $W.$  (So all elements in $\mathbf{a}$ are the same, except maybe $\mathbf{a}(s)$.) Therefore if $X\in \G$, then $X$ is a simple bipartite graph $(U,W;E(X))$ and $E(X)$ does not contain any element from $\F.$ The edge set $E(\G)$  corresponds to the possible swap operations.

Sinclair's multicommodity flow method  defines a bunch of paths (consecutive sequences of swaps) for each realizations pair  $X$ and $Y$ which transform realization $X$ into $Y.$

So consider two realizations $X\in \G$ and $Y\in \G$, and consider the symmetric difference $\Delta = E(X) \Delta E(Y)$. In the bipartite graph $\Theta=(U,W;\Delta)$ for each vertex $v$ the number of adjacent $X$-edges ($=E(X)\setminus E(Y)$)  and the number of the adjacent $Y$-edges are the same. Therefore, due to Euler classical reasoning, it can be decomposed into {\em alternating circuits.}

The simplified Sinclair's multicommodity path method consists of two phases: In {\bf Phase 1}  we decompose the symmetric difference $\Delta$ into alternating circuits on all possible ways. In each cases we get an ordered sequence $W_1,W_2,\dots, W_\kappa$ of circuits. (Usually there are a huge number of different decompositions.) Each circuit is endorsed with  a fixed cyclic order.

 In {\bf Phase 2}  each circuit $W_i$ from the (ordered) decomposition  derives one unique alternating cycles decomposition: $W_i = C^i_1,C^i_2,\dots, C^i_{k_i}.$ This decomposition is fully determined by the circuit and its well defined  edge order. (Both construction algorithms are fully described in Section 5 of the paper \cite{MES}, we do not discuss them here.)

The ordered circuit decomposition together with the ordered cycle decompositions of all circuits altogether provide a well defined ordered cycle decomposition $C_1,\ldots C_\ell$ of $\Delta.$

This ordered cycle decomposition determines $\ell-1$ realizations $H_1,\ldots H_{\ell-1}$ with the following property: if we use the notations $H_0=X$ and $H_{\ell}=Y$ then for each $j=0,\ldots, \ell-1$ we have $E(H_j) \Delta E(H_{j+1}) = C_{j+1}.$ (It is important to recognize that till this point we did not process even one swap operation! We just identified $\ell-1$ realizations which will be along our canonical path.)

We will define a unique canonical path from $X$ to $Y$ determined by this circuit decomposition which uses these realizations $H_j$ as {\em milestones} along the path.  The canonical path will be $ X=G_0, \dots,G_i,\dots, G_{m}=Y$ where each $G_i$ can be derived from $G_{i-1}$ with one valid swap operation, where we must have the following property:  there are some increasing indices $0<n_1<n_2<\dots < n_\ell$ such that we have $G_{n_i}=H_i$. This, together the definitions of $H_i$ means that
$$
E(H_i)=E(X)\bigtriangleup\left(\bigcup_{i'<i}E(C_{i'})\right).
$$
The canonical path we are looking for has two important further properties: for each $i<\ell$ the constructed path $H_i=G'_0,G'_1,\dots, G'_{m'}=H_{i+1}$ between $G_{n_i}$ and $G_{n_{i+1}}$ must satisfies that
\begin{enumerate}
\item[$(\Theta)$] $m' \le c\cdot |C_i|$ for a suitable constant $c;$
\item[$(\Omega)$] for each $j$ there is $K_j\in V(\G) $ such that $\mathfrak{d}\left (M_X +M_Y -M_{G'_{j}},M_{K_j} \right )\le \Omega_2$,
\end{enumerate}
where the notations $M_G$ stands for the usual {\bf bipartite adjacency matrix} of $G$ (this will be defined in details at the beginning of the next section), and $\mathfrak{d}$ stands for the Hamming distance of two matrices of the same form, finally $\Omega_2$ is a small constant.

The current value of the auxiliary matrix $M_X +M_Y -M_{G'_{j}}$ together with the symmetric difference $\Delta,$ furthermore a small (polynomial) size parameter set, finally the vertices in $\G$ on which  the canonical path under investigation goes through uniquely determine the vertices $X, Y$ and the path itself. Therefore it can be used to control certain features of the canonical path system. If the overall number of these auxiliary matrices are small (their number is smaller than a small polynomial of $n$ multiplied with the number of possible realizations - as it is ensured by $(\Omega)$), then - as it was proved in \cite{MES} - our Markov chain is rapidly mixing.

So in {\bf Phase 2} we have to build up our swap sequence between $H_i$ and $H_{i+1}$ for all values $i$ taking care for conditions ($\Theta$) and ($\Omega$). This will happen in the next Section.

\section{The construction of swap sequences between consecutive "milestones"}\label{sec:miles}
\noindent Now we are going to implement our plan described above. At first we introduce some shorthand. Instead of $H_i$ and $H_{i+1}$ we will use the names $G$ and $G'.$
These two graphs have almost the same edge set. More precisely
\begin{eqnarray*}
\bigl(E(G) \setminus (C_i \cap E(X))\bigr ) \cup (C_i \cap E(Y)) = E(G') \\
\bigl ( E(G') \setminus (C_i \cap E(Y)) \bigr ) \cup (C_i \cap E(X)) = E(G).
\end{eqnarray*}
Of course  $E(G) \Delta E(G') = C_i$ also holds. We  refer for the elements of $C_i \cap E(X)$ as $X$-edges while the others  are the $Y$-edges. We  denote the cycle itself as $\C,$ it has  $2\ell$ edges and its vertices are
$u_1,w_1,u_2,w_2,\ldots ,u_\ell, w_\ell.$ Since $\C$ has at least four vertices, therefore we may assume that $u_1 \ne s$ (so $u_1$ is not the center of the forbidden star). Finally w.l.o.g. we may assume that the chord $u_1w_1$ is an $Y$-edge (and, of course, $w_\ell u_1$  is an $X$-edge).

We are going to construct one by one the realizations $G_j'$. We build our canonical path from $G$ toward $G'$ and at any particular point the last constructed realization is denoted by $Z.$  (At the beginning of the process we have $Z=G$.) We are looking for the next realization, denoted by $Z'.$

Before we continue the discussion of  the canonical path system, we have to introduce our control mechanism, mentioned in condition ($\Omega$). This {\em auxiliary} structure originally was introduced by Kannan, Tetali and Vempala in \cite{KTV97}:

For any particular realization $G$ from $\G$ the matrix $M_G$ denotes the {\em adjacency matrix} of the bipartite realization $G$ where the columns and rows are indexed by the vertices of $U$ and $W$ resp. (Therefore the column sums are the same in each realization, except perhaps at column $s$.) Our indexing method is a bit unusual:  the columns are numbered from left to right while the rows are numbered from bottom to the top. (Like in the Cartesian coordinate system.)  This matrix is not necessarily symmetric, and elements $M_{i,i}$  can be different from 0.

For example if we consider the submatrix in $M_G$ spanned by $u_1,\ldots, u_\ell$ and $w_1, \ldots, w_\ell$ then we have $M_G(i,i)=0$ for $i=1,\ldots,\ell$, while  $M_G(i,i-1)=1$ (for $i=2,\ldots, \ell$) and $M_G(1,\ell)=1$. (So the first value gives the column, the second one gives the row.) The non-chords between vertices in the same vertex class are not considered at all, while non-chords which are forbidden are denoted by $\maltese.$ As it is clear from the previous sentence, we will identify each  chord or non-chord with the corresponding position in the matrix.

\medskip\noindent
Our auxiliary structure is the matrix
\begin{displaymath}
\widehat M(X+Y-Z) = M_X + M_Y - M_Z.
\end{displaymath}
By definition, each entry of a bipartite adjacency matrix is $0$  or $1$ (or  $\maltese$). Therefore only $-1,0,1,2$ can be the "meaningful" entries of $\widehat M.$ An entry is $-1$ if the edge is missing from both $X$ and $Y$ but it exists in $Z.$ This is $2$ if the edge is missing from $Z$ but exists in both $X$ and $Y.$ It is $1$ if the edge exists in all three graphs ($X,Y,Z$) or it is there only in one of $X$ and $Y$ but not in $Z.$ Finally it is $0$ if the edge is missing from all three graphs, or the edge exists in exactly one of $X$ and $Y$ and in $Z.$ (Therefore if an edge exists in exactly one of $X$ and $Y$ then the corresponding chord in $\widehat M$  is always $0$ or $1$.)   One more important, but easy fact is the following:
\begin{observ}
The row and column sums of $\widehat M(X+Y-Z)$ are the same as row and column
sums in $M_X$ (or $M_Y$ or $M_Z$). \qed
\end{observ}

\medskip\noindent Next we will determine the swap sequence between $G$ and $G'$ through an iterative algorithm. At the first iteration we check, step by step, the positions $(u_1, w_2), (u_1, w_3), \ldots, (u_1,w_\ell)$ and take the smallest $j$ for which $(u_1,w_i)$ is an actual edge in $G.$ Since $(u_1,w_\ell)$ is an edge, therefore such $i$ always exists. So we may face to the following configuration:

\begin{figure}[h!]
\center{
\begin{tikzpicture}[scale=0.35]
\node at (4,0) [shape=circle,draw] (p1) {$u_1$};
\node at (8,-1) [shape=circle,draw] (p2) {$w_1$};
\node at (12,0) [shape=circle,draw] (p3) {$u_2$};
\node at (15,4) [shape=circle,draw] (p4) {$w_2$};
\node at (15,9) [shape=circle,draw] (p5) {$w_{i-1} $};
\node at (12,12) [shape=circle,draw] (p6) {$u_i$};
\node at (8,13) [shape=circle,draw] (p7) {$w_i$};
\node at (4,12) [shape=circle,draw] (p8) {$\phantom{u_i} $};
\node at (0,9) [shape=circle,draw] (p9) {$u_\ell$};
\node at (0,4) [shape=circle,draw] (p10) {$w_\ell$};

\draw [very thick,dotted,red] (p1) -- (p2);
\draw [very thick] (p2) -- (p3);
\draw [very thick, dotted,red] (p3) -- (p4);
\draw [very thick,snake=snake,dotted] (p4) -- (p5);
\draw [very thick] (p5) -- (p6);
\draw [very thick, dotted,red] (p6) -- (p7);
\draw [very thick, dotted,red] (p7) -- (p8);
\draw [very thick,snake=snake,dotted] (p8) -- (p9);
\draw [very thick, dotted,red] (p9) -- (p10);
\draw [very thick]  (p10) -- (p1);
\draw [very thick,dotted,red] (p1) -- (p4);
\draw [very thick,dashed,gray!50] (p1) -- (p5);
\draw [very thick,dashed,gray!50] (p1) -- (p6);
\draw [very thick] (p1) -- (p7);

\draw [very thick] (-8,-3) -- (-6,-3);
\node at (-4,-3) {edge};

\draw [very thick,dotted,red] (-1.5,-3) -- (0.5,-3);
\node at (5.5,-3) {chord, non-edge};

\draw [very thick,dashed,gray!50] (11,-3) -- (13,-3);
\node at (16,-3) {non-chord};

\draw [very thick,snake=snake,dotted] (20,-3) -- (22,-3);
\node at (24,-3) {unknown};

\end{tikzpicture}
}
\caption{Sweeping a cycle}\label{f:sweep}
\end{figure}
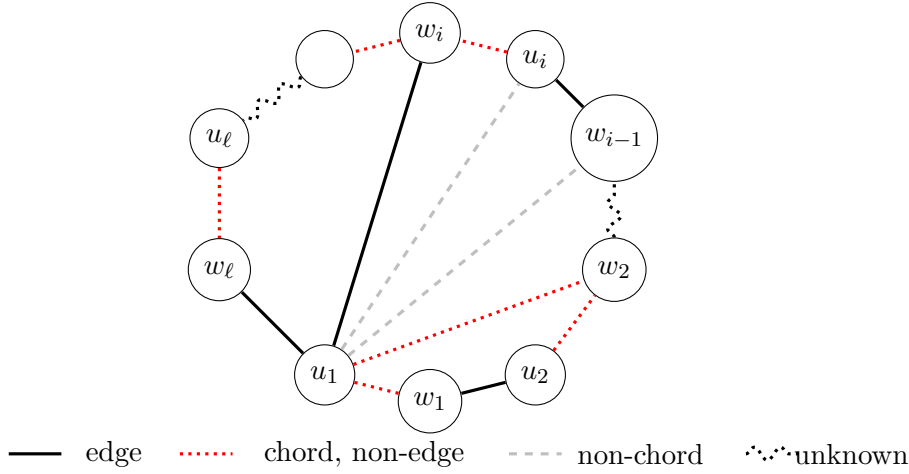
\noindent We will call this  $(u_1,w_i)$ chord as {\bf start-chord} of our current sub-process and $u_1w_1$ is the {\bf end-chord}. We will {\em sweep} the alternating chords along the cycle from the start-edge $w_iu_i$ (non-edge), $u_iw_{i-1}$ (an edge) toward the end-edge $w_1u_1$ (non-edge) -- switching their status in twos and fours.  We check positions $u_1w_{i-1}, u_1w_{i-2}$ (all are non-edges) and choose the first chord among them, we will call it the {\bf current-chord}. (Since $u_1 \ne s$ therefore we never have to check more than two edges to find the first chord, and we need only one times to check two, since there is at most  one non-chord adjacent to $u_1.$)

\smallskip\noindent {\bf Case 1}: As we just explained the typical situation is that the current-chord is the "next" one, so when we start this is typically $u_1w_{i-1}.$ Assume that this is a chord. Then we can proceed with the swap operation $w_{i-1}u_i, w_iu_1 \Rightarrow u_1w_{i-1}, u_iw_i.$ We just produced the first "new" realization in our sequence, this is $G'_1.$  For the next swap operation this will be our new current realization. This operation  will be called a {\bf single-step}.

In a realization $Z$ we will call a chord {\bf bad}, if its current status (being edge or non-edge) is different from its status in $G$ (or, what is the same, in $G'$, since they differ only on the chords along the cycle $\C$). After the previous swap, we have two bad chords in $G'_1,$ namely $u_1w_{i-1}$ and $w_iu_1$.

Consider now the auxiliary matrix $\widehat M(X+Y-Z)$ (here $Z=G'_1$). As we  saw earlier, for each position outside the chords in $\C$ the status of that particular position in $Z$ is the same as in $X$ or $Y$ or in both. Accordingly, the corresponding matrix value is $0$ or $1.$ We call a position  {\bf bad} in $\widehat M$ if this value is $-1$ or $2$. (A bad position in $\widehat M$ always corresponds to a bad chord.) Since in Case 1 we switch the start-chord  into non-edge, it may become $2$ in $\widehat M.$ (In case if in both $X$ and $Y$ it is an edge. Otherwise it is $0$ or $1$, so in that case it is not a bad position.) The current-chord turned into an edge. If it is non-edge in both $X$ and $Y$ then the value becomes $-1$, otherwise it does not become a bad position. After this single-step, we have at most two bad positions in the matrix, at most  one position with $2$-value  and at most one with  $-1$-value.

\smallskip\noindent {\bf Case 2}: If the position "below" the start-chord is a non-chord, then we cannot produce the previous swap. Then,  however, the non-edge $u_1w_{i-2}$ is the current-chord. For sake of simplicity we assume that $i-2=2$ so we are in Figure \ref{f:sweep}. Consider now the alternating  $C_6$ cycle:  $u_1,w_2, u_3,w_3, u_4,w_4.$ It has altogether three vertex pairs which may be chords. We know already that $u_1w_3$ is a non-chord. If none of the three is chord, then this is an $\F$-compatible circular $C_6$-swap - and accordingly to the definitions we can swap it in one step. Again, we found the valid swap $w_2u_3, w_3u_4, w_4u_1 \Rightarrow u_1w_2, u_3w_3, u_4w_4.$ After that we again have 2 bad chords, namely $u_1w_2$ and $w_4 u_1,$ and together we have at most two bad positions in the new $\widehat M(X+Y-Z)$ with at most one $2$-value and at most one $-1$-value.

Finally if one position, say $w_2u_4,$  is a chord then we can process this $C_6$ with two swap operations. If this chord is, say, an actual edge, then we swap $w_2u_4, w_4u_1 \Rightarrow u_1w_2, u_4w_4.$ After this we can take care for the $w_2,u_3,w_3,u_4$ cycle. Along this sequence we never create more, than 3 bad chords: the first swap makes chords $w_2u_4, w_4u_1$ and $u_1w_2$  bad ones, and the second "cures" $w_2u_4$ but does not touch  $u_1w_2$ and $w_4u_1.$ So along this swap sequence we have 3 bad chords, at the end we have only 2. On the other hand, if the chord $w_2u_4$ is not an edge, then we can swap $w_2u_3, w_3u_4 \Rightarrow u_3w_3, u_4w_2$, creating one bad edge, then taking care the four cycle $u_1,w_2,u_4,w_4$ we "cure" $w_2u_4$ but we switch $u_1w_2$ and $w_4u_1$ into bad chords. We finished our {\bf double-step} along the cycle.

In a double-step we make at most  three bad chords. When the first swap uses three chords along the cycle then we may have at most one bad chord (with $\widehat M$-value $0$ or $-1$) and then the next swap switches back the chord into its original status, and makes two new bad chords (with at most one $2$-value and one $-1$-value). When the first swap uses only one chord from the cycle, then it makes three bad chords (changing two chords into non-edge and one into edge), therefore it may make at most two $2$-values and one $-1$-value. After the second swap there will be only two bad chords, with at most one $2$-value, and at most one $-1$-value.

When only the third position corresponds to a chord in our $C_6$ then after the first swap we may have two $-1$-values and one $2$-value. However, again after the next swap we will have at most one of both types.
\begin{remark}\label{r:distance}
When two realizations are one swap apart (so they are adjacent in $\G$) then we say that their auxiliary matrices are at swap-distance one. Since one swap changes four positions of the matrix, therefore the Hamming distance of these matrices is 4.
\end{remark}

\medskip\noindent
Finishing our single- or double-step the previous current-chord becomes the new start-chord and we look for the new current-chord. Then we repeat our procedure. There is only one important point to be mentioned: along the step, the start-chord switches back into its original status, so it will not be a bad chord anymore. So even if we face a double-step the number of bad chords never will be bigger than three (together with the chord $w_iu_1$ which is still in the wrong status, so it is bad), and we have always at most two $2$-values and at most one $-1$-value in $\widehat M(X+Y-Z).$

When our current-chord becomes to $w_1u_2$ then the last step will switch back the last start-chord into its correct status, and the last current-chord cannot be in bad status. So, when we finish our sweep from $u_1w_i$ to $w_1u_1$ at the end we will have only one bad chord (with a possible $2$-value in $\widehat M$). This concludes the first iteration of our algorithm.

\bigskip\noindent
For the next iteration we seeks a new start-chord between $w_iu_1$ and $w_\ell u_1$  and chord $w_iu_1$ becomes the new end-chord. We will repeat our sweeping process for this setup, and we will repeat it as long as all chords will be processed, so we fond the entire realization sequence from $G$ to $G'.$ If in the first sweep we had a double-step, then it will never occur later, so altogether with the bad (new) end-chord we never have more than three bad chords, with at most two $2$-values and at most one $-1$-value.

However, if the double-step occurs sometimes later, for example in the second sweep, then we face to the following situation: if we perform a circular $C_6$-swap, then there cannot be any problem. So we may assume that there is a chord in our $C_6$, suitable for a swap. If this chord is a non-edge, then the swap around it produces one bad chord, and at most one bad position in $\widehat M.$ The only remaining case when that chord is an edge. After the first swap there will be four bad chords, and there may be at most three $2$-values and at most one $-1$ value. However after the next swap (finishing the double step) we annihilate one of the $2$-values, and after that swap there are at most two $2$-values and at most $-1$-value along the entire swap sequence. When we finish our second sweep, then chord $w_iu_1$ will be switched back into its original status, it will not be bad anymore.

We apply iteratively the same algorithm, and after at most $\ell$ sweep sequence, we will process the entire cycle $\C.$ This finishes the construction of the required swap sequence (and the required realization sequence). \qed

\medskip\noindent Meanwhile we also proved the following important observation:
\begin{lemma}\label{th:bad}
Along our procedure each occurring auxiliary matrix $\widehat M(X+Y-Z)$ is at most swap-distance one from a matrix with at most three bad positions: with at most two $2$-values and with at most one $-1$-value in the same column, which does not coincide with the center of the forbidden star.
\end{lemma}

\section{The analysis of the swap sequences between "milestones"}\label{sec:anal}
\noindent
What remains is to show that the defined swap sequences between $H_i$ and $H_{i+1}$ satisfy the properties $(\Theta)$ and $(\Omega)$ of the simplified Sinclair's method. The first one is easier to see, since we can process a cycle of length $2\ell$ in $\ell-1$ swaps. Therefore the derived constant $c$ in $(\Theta)$ is actually 1.

We introduce the {\bf switch} operation on $0/1$ matrices with forbidden positions: we fix the four corners of a submatrix (none of them is forbidden), and we add $1$ to two corners in a diagonal, and add $-1$ to the corners on the other diagonal. This operation clearly does not change the column and row sums of the matrix. For example if we consider the matrix $M_G$ of a realization of our $\bd ^\F$ and make a valid swap operation, than it looks like  as a switch in this matrix. The next statement is trivial but very useful:
\begin{lemma}\label{th:Hamming}
If two  matrices have {\bf switch-distance} 1, then their Hamming distance is $4$. Consequently if the switch-distance is $c$ then the Hamming distance is bounded by $4c.$
\end{lemma}
\noindent We will prove now that property $(\Omega)$ holds for our auxiliary matrices:
\begin{theorem}\label{th:omega}
For any realizations $X$ and $Y$ furthermore for any realization $Z$ on a swap sequence from $X$ to $Y$ there exists a realization $K$ such that
$$
\mathfrak{d}\left (\widehat M(X+Y-Z), M_K \right ) \le 16.
$$
\end{theorem}
\noindent Due to Lemmas \ref{th:bad} and \ref{th:Hamming} it is enough to show that:
\begin{lemma}\label{th:switch}
Any matrix $\widehat M(X+Y-Z)$ with constant column sums (this not necessarily holds for the center of the forbidden star)  and at most three bad positions  (where there are at most two $2$-values and at most one $-1$-value) can be transformed into a valid $M_K$ adjacency matrix with at most three switch operations.
\end{lemma}
\begin{proof}
Consider now a certain $\widehat M$ which is not necessarily a valid adjacency matrix of a realization. We will show pictures about the submatrix in this matrix which describes the current alternating cycle $\C.$ We choose a submatrix, where the center $s$ of the forbidden star is in the first column. (We choose this submatrix as an illustration tool, but we still consider the entire matrix to work with.) We know that this matrix contains at most two $2$-values and at most one $1$-value. All these positions are adjacent to the center $u_1$  of our sweeping sequence (see Figure \ref{f:sweep}), so they are in the same column.

For  simplicity from now on we will denote the center of the sweep as well the column  $u.$ The forbidden positions are denoted with $\maltese$. Any column (except column 1) may contain at most one of them, and any row may contain at most two of them. Finally in our pictures the character $\diamond$ stands for a character which we are not interested in. That is, it can be $0$ or $1$ or $\maltese.$

We will distinguish cases, depending on the occurring of values $2$ and $-1.$

\medskip\noindent {\bf{Case 1.}} Column $u$ has one bad position, which can be $-1$ or $ 2$, or it has two 2-values. Consider at first the subcase when $\widehat M[uw] = -1$. By definition that means that  chord $uw$ is an edge in $Z$ but non-edge in both $X$ and $Y.$ So vertex $w \in W$ has at least one adjacent edge, therefore the row-sum in its row is at least $1.$ Therefore there are at least two positions in row $w$ with entries $1$. They are in column $u_1$ and $u_2.$  At least one of them, say $u_1$, differs from $s$. Since the column sums are constant, therefore there exists at least two rows $w_1$ such that $\widehat M[uw_1]=1$ while $\widehat M[u_1w_1]=0$ or $\maltese$. However, there can be at most one forbidden position in $u_1$, so in at least one of the rows, the entry is $0$. Using these positions for the corresponding switch it eliminates the bad position without creating a new one. (See Figure \ref{f:ketto}.)
\begin{figure}[h!]
$$
\begin{pmatrix}
 \maltese &\diamond   &\diamond&\diamond &\diamond&\diamond \\
& & \vdots &  & \vdots &    \\
\maltese &\diamond    &\diamond&   \maltese    &\diamond&\diamond  \\
 &  & \vdots &  & \vdots & \\
\maltese & \mathbf{1}    &\diamond&   \mathbf{0}    &\diamond&\diamond \\
& & \vdots &  & \vdots &    \\
\diamond & \mathbf{-1}      &\diamond&   \mathbf{1}    &\diamond & \mathbf{1}      \\
& & \vdots &  & \vdots &    \\
\end{pmatrix}
\qquad
\Rightarrow \qquad
\begin{pmatrix}
 \maltese &\diamond   &\diamond&\diamond &\diamond&\diamond \\
& & \vdots &  & \vdots &    \\
\maltese &\diamond     &\diamond&   \maltese    &\diamond&\diamond  \\
 &  & \vdots &  & \vdots & \\
\maltese & \mathbf{0}    &\diamond&   \mathbf{1}    &\diamond&\diamond \\
& & \vdots &  & \vdots &    \\
\diamond & \mathbf{0}      &\diamond&   \mathbf{0}    &\diamond & \mathbf{1}      \\
& & \vdots &  & \vdots &    \\
\end{pmatrix}
$$
\caption{Case 1\qquad $\maltese=$ forbidden \quad$\diamond=0/1/\maltese$ }\label{f:ketto}
\end{figure}

\noindent Before we continue, we prove an important observation:\\
{\bf Observation} If $w$ belongs to the alternating cycle $\C$ and $\widehat M[uw]=2$ then row $w$ contains at least two $0$-values.

Indeed, there are $\alpha$ forbidden chords in row $w.$  Since $w$ is in an alternating cycle, therefore $d(w)\le |U|-\alpha-1.$ Therefore the sum of row $w$ in $\widehat M(X+Y-Z) \le |U|-\alpha-1.$ But it contains a $2$ and it does not contain -1 therefore there are at least two $0$'s in it. \qed

When the single bad value in $\widehat M$ is $2$ then, due to our previous Observation, in its row there are two $0$'s. And with them one can repeat the reasoning which we used about the unique $-1$-value.

Finally, when there are two $2$-values which raises a very similar situation. Here we can do  the same procedure independently on both rows. In this case, however, we need two switch operations.

\noindent {\bf{Case 2.}} Here we assume that there is one $2$-value and one $-1$-value in column $u.$ For example $\widehat M[uw_1]=2$ and  $\widehat M[uw_2]=-1$. Again, in row $w_2$ there are at least two $1$-values. \\
{\it Case 2a} Assume at first that we have $u_1\in U$ s.t.  $\widehat M[u_1w_2]=1 $ and $\widehat M[uw_1]\ne \maltese.$ Then the corresponding switch will produce $\widehat M[u_1w_1]= 1/2$ while the three positions are $0$ or $1.$ (See Figure \ref{f:harom}.)   If now $\widehat M[u_1w_1]= 2$ then we are back to Case 1, and one more switch eliminates the last bad position as well. So we needed at most two switches.

\begin{figure}[h!]
$$
\begin{pmatrix}
 \maltese &\diamond   &\diamond&\diamond &\diamond&\diamond \\
& & \vdots &  & \vdots &    \\
\maltese & \diamond    &\diamond&   \maltese    &\diamond&\diamond  \\
 &  & \vdots &  & \vdots & \\
\maltese & \mathbf{0/1}    &\diamond&   \mathbf{2}    &\diamond&\diamond \\
& & \vdots &  & \vdots &    \\
\diamond & \mathbf{1}      &\diamond&   \mathbf{-1}    &\diamond & \diamond    \\
& & \vdots &  & \vdots &    \\
\end{pmatrix}
\qquad
\Rightarrow \qquad
\begin{pmatrix}
 \maltese &\diamond   &\diamond&\diamond &\diamond&\diamond\\
& & \vdots &  & \vdots &    \\
\maltese & \diamond    &\diamond&   \maltese    &\diamond&\diamond  \\
 &  & \vdots &  & \vdots & \\
\maltese & \mathbf{1/2}    &\diamond&   \mathbf{1}    &\diamond&\diamond\\
& & \vdots &  & \vdots &    \\
\diamond & \mathbf{0}      &\diamond&   \mathbf{0}    &\diamond & \diamond      \\
& & \vdots &  & \vdots &    \\
\end{pmatrix}
$$
\caption{Case 2a \qquad $\maltese=$ forbidden \quad$\diamond=0/1/\maltese$ }\label{f:harom}
\end{figure}

\begin{figure}[h!]
$$
\begin{pmatrix}
 \maltese &\diamond   &\diamond&\diamond &\diamond&\diamond\\
& & \vdots &  & \vdots &    \\
\maltese & \diamond      &\diamond&   \maltese    &\diamond&\diamond  \\
 &  & \vdots &  & \vdots & \\
\maltese & \mathbf{0}    & \diamond &   \mathbf{2}    &\diamond&\maltese\\
& & \vdots &  & \vdots &    \\
\mathbf{1} & \mathbf{0}      &\diamond &   \mathbf{-1}    &\diamond & \mathbf{1}  \\
& & \vdots &  & \vdots &    \\
\end{pmatrix}
\qquad
\Rightarrow \qquad
\begin{pmatrix}
 \maltese &\diamond   &\diamond&\diamond &\diamond&\diamond\\
& & \vdots &  & \vdots &    \\
\maltese &\diamond    &\diamond&   \maltese    &\diamond&\diamond  \\
 &  & \vdots &  & \vdots & \\
\maltese & \mathbf{1}    & \diamond &   \mathbf{1}    &\diamond&\maltese\\
& & \vdots &  & \vdots &    \\
\mathbf{1} & \mathbf{-1}      &\diamond &   \mathbf{0}    &\diamond & \mathbf{1}  \\
& & \vdots &  & \vdots &    \\
\end{pmatrix}
$$
\caption{Case 2b\qquad $\maltese=$ forbidden \quad$\diamond=0/1/\maltese$ }\label{f:negy}
\end{figure}
\noindent {\it Case 2b} It can happen, that there are only two $1$-values in row $w_2$ and both are facing with forbidden positions in row $w_1.$ Then at least one $0$ in row $w_2$ faces a chord in row $w_2.$ (See Figure \ref{f:negy}) The appropriate switch kills 2 bad chords and can make at most one $-1$-value. We are ready or we are back to Case 1.

\medskip
\noindent {\bf{Case 3.}} Finally suppose that there are three bad positions, two $2$-values at positions $uw_1$ and $uw_2$ and one $-1$-value at position $uw_3.$ Now both rows $w_1$ and $w_2$ contains at least two $0$'s. If any of them faces a $1$ in row $w_3$ then an appropriate switch annihilates one $2$ and one $-1$ and does not create new bad position. We are back to Case 1. Altogether we need two switches.

If this is not the case then we consider the following: assume that $\widehat M[u_1w_1]=0.$ Since the column sums are the same, and we assumed that  $\widehat M[u_1w_3]=0$ therefore there exists a row $w_4$ s.t. $\widehat M[u_1w_4]=1$ while $\widehat M[uw_4]=0.$ Then we can switch off this $2$-value without making a new bad position. After that we are back to Case 2. Altogether this requires at most three switches. We finished the proof of Lemma \ref{th:switch}.

If this is not the case then we consider the following: assume that $\widehat M[u_1w_1]=0.$
The column sums are the same, and we assumed that $\widehat M[u_1w_3]$ $=0$ or $\maltese$. Therefore the difference between column sums in $u$ and $u_1$ is $1$ due to rows $w_1$ and $w_3$, and the difference increase at least $1$ for row $w_2$, where against a $2$-value  in column $u$ there is either $1$ or $0$ in column $u_1$.  Therefore there exists at least two further rows, where there is a $1$ in column $u_1$ against a $0$ or $\maltese$ in column $u$. Since column $u$ can contain at most one $\maltese$, one of the rows must contain a $0$. Let it be denoted by $w_4$. Hence $\widehat M[u_1w_4]=1$ while $\widehat M[uw_4]=0.$ Then we can switch off this $2$-value without making a new bad position. After that we are back to Case 2. Altogether this requires at most three switches. We finished the proof of Lemma \ref{th:switch}.
\end{proof}
In turn this proves Theorem \ref{th:omega}, so our Markov chain is rapidly mixing as Theorem \ref{th:main} stated.

\section{Self-reduced counting problem}\label{sec:counting}
A decision problem is in NP if a non-deterministic Turing Machine can solve it in polynomial time. An equivalent definition is that there exists a witness proving the {\bf yes} answer to the question which witness can be verified in polynomial time. A counting problem is in \#P if it asks for the number of those witnesses of a problem from NP that can be verified in polynomial time (it might happen that not all witnesses are verifiable in polynomial time).

Two complexity classes, FPRAS and FPAUS, concern the approximability of counting problems. Here we give only narrative descriptions of these complexity classes, the detailed definitions can be found, for example, in \cite{JVV}.

A counting problem from \#P is in FPRAS (Fully Polynomial Randomized Approximation Scheme) if the number of solutions can be quickly estimated with a randomized algorithm such that the estimation has a small relative error with very high probability.

A counting problem from \#P is in FPAUS (Fully Polynomial Almost Uniform Sampler) if the solutions can be sampled quickly with a randomized algorithm that generates samples following a distribution being very close to the uniform one.

It is easy to see that a counting problem is in FPAUS if there is a rapidly mixing Markov chain for which
\begin{itemize}
\item a starting state can be generated in polynomial running time;
\item one step in the Markov chain can be conducted in polynomial running time; and
\item the relaxation time of the Markov chain grows only polynomially with the size of the
problem.
\end{itemize}
The Markov chain we gave the star+factor problem satisfies all these requirements.

Jerrum, Valiant and Vazirani proved that any \emph{self-reducible   counting problem} is in FPRAS iff it is in FPAUS \cite{JVV}.  A counting problem is self-reducible if the solutions for any problem instance can be generated recursively such that after each step in the recursion, the remaining task is another problem instance from the same problem, and the number of possible branches at each recursion step is polynomially bounded by the size of the problem instance.

Clearly, a graph with prescribed degree sequence can be built recursively by telling the neighbors of a node at each step, then removing the node in question and reducing the degrees of the selected neighbors. However, this type of recursion does not satisfy all the
requirement for being self-reducible since there might be exponentially many possibilities how to select the neighbors of a given vertex.

On the other hand, the degree sequence problem with a forbidden one factor and star tree is a self-reducible counting problem. Indeed, consider the center of the (possibly empty) star, $s \in U$, and the vertex $v \in V$ with the smallest index for which $(s,v)$ is a chord. Any  solution for the current problem instance belongs to one of the following two cases:
\begin{itemize}
\item The chord $(s,v)$ is not present in the solution. In that   case, extend the size of the star by adding chord $(s,v)$ to the   forbidden set, and do not change the degrees. This is another   problem instance from the star+factor problem, whose solutions are   the continuations of the original problem belonging to this case.
\item The chord $(s,v)$ is present in the solution. In that case,   extend the size of the star by adding chord $(s,v)$ to the forbidden   set, and decrease both $d_s$ and $d_v$ by one. The new degree sequence is still a half-regular, bipartite star+factor restricted degree sequence, and the solutions of this new problem extended with the previously decided step provide solutions of the original problem.
\end{itemize}
Since the star+factor counting problem is a self reducible counting problem, it is in FPRAS as it is in FPAUS.

\bigskip\noindent
We finish this paper with a short analysis of the connections between our approach and the paper \cite{JSV} of Jerrum, Sinclair and Vigoda. Their seminal result from 2004 solved the uniform sampling problem of perfect 1-factors of a given graph. As their Corollary 8.1 pointed out this method can be applied for uniform sampling of the set of all possible realizations of a given $f$-factor of a complete graph. It also proves that the problem is in FPRAS (and in FPAUS as well).

Since the restricted degree sequence problem in general is equivalent to the $f$-factor problem, therefore our star+factor RDS problem is only a special case of the $f$-factor problem, so the JSV result applies for it. This describes the similarity.

The important differences lay in the swap operations applied in the JSV method and in the Kannan-Tetali-Vempala's Markov chain. In the JSV method a special graph $\mathfrak{G}$ is introduced for the sampling via Tutte's gadgets. Then the swap operations are working on the graph $\mathfrak{G}$ with the unintended result that for a (sometimes long) sequence of swaps does not change at all the generated $f$-factor. Combining this issue with the known relative slow mixing time of the Jerrum-Sinclair-Vigoda's Markov chain, the resulted approach in not suitable for any practical application.

The KTV Markov chain operates in the original graph and each jump provides a new realization of the original degree sequence problem. The KTV Markov chain is  presumably much faster than the JSV chain, furthermore the JSV theorem does not proves the fast mixing nature of the KTV chain. Similarly it does not prove that the KTV chain provides a fast approximate counting algorithm.

\appendix
\section*{Appendix: the simplified Sinclair's method}
In this paper the fast mixing nature of our MCMC method was proved through the  application of the simplified Sinclair's  method, developed in \cite{MES}. To do so properly it requires a slight generalization of the original method.

The method takes two realizations $X$ and $Y$ of the same degree sequence. It  considers all possible ordered circuit decompositions of the symmetric difference of the edge sets, then it uniquely decomposes each such decomposition into  an ordered sequence $\C=C_1, \ldots, C_m$ of oriented cycles. Based on this latter decomposition the method determines a well defined unique path between $X$ and $Y$  in the Markov chain $\G.$

For that end the method defines first a sequence of "milestones". These are different realizations $X=H_0, H_1, \ldots, H_{m-1}, H_m=Y$  of the degree sequence where the edge set of any two consecutive realizations  $H_{i-1}, H_i$ differ exactly in the edges along the cycle  $C_i.$ (Until this point no swap operation happened.)

In the next phase for any particular $i=0,\ldots, m-1$ the method determines a sequence of valid swap operations transforming $H_{i-1}$ into $H_i$ - describing a unique path $Z_0,Z_1,\ldots Z_{\ell}$ between $H_{i-1}$ and $H_i$ in the Markov chain $\G.$ This sequence of course heavily depends on the available swap operations. In paper \cite{MES} these are the usual (bipartite) swap operations. In the current paper these are the restricted swap operations. These operations, while exchanging chords in the realizations along  the alternating cycle $C_i$, also use some further chords. Therefore the edge set of any $Z_i$ is not completely contained by $E(X)\cup E(Y),$  there exist a small number of edges in $Z_i$ which are non-edges in $X$ and in $Y$, or non-edges in $Z_i$ but edges in $X$ and $Y.$ If $Z_i$ is between the milestones $H_m$ and $H_{m+1}$, then $C_j$ for $j\ne m$ alternates in $Z_i$, and $C_i$ alternates with a "small error": there is a very small number of vertices where the alternation does not hold.

Along the process the simplified Sinclair's method requires (see the paper \cite{MES}, Section 5, (F)(c) ) that this number must be small. In the original application this number is actually one. Here, as we saw in Section \ref{sec:miles}, this number is three: that many bad chords may occur after any particular RSO. As we saw all these chords are adjacent to the same vertex $u_1.$

These numbers are used by the method to determine the size of a parameter set $\B$. This parameters set must have a polynomial size. When we have one bad chord, then  it is determined by its end points - there are at most $n^2$ possibilities for them. This provides an $n^2$ multiplicative factor to the size of $\B.$ When we have at most three bad chords, then they can be chosen at most $n^4$ ways: point $u_1$ is fixed ($n$ different choices), while the other three end points can be chosen at most $n^3$ independent ways. Altogether it provides an at most $n^4$ multiplicative factor to the size of $\B.$ This remark finishes the proof of the simplified Sinclair's method for the case of these restricted swap operations.

\begin{thebibliography}{99}

\bibitem{CDG07} Cooper, C. - Dyer, M. - Greenhill, C.: Sampling regular graphs and a peer-to-peer network, {\sl Comb. Prob. Comp.} {\bf 16} (4) (2007), 557--593.

\bibitem{EG60} Erd\H{o}s, Paul - Gallai, T.: Gr\'afok el\H{o}\'\i rt fok\'u pontokkal (Graphs with prescribed degree of vertices), {\sl Mat. Lapok}, {\bf 11} (1960), 264--274. (in Hungarian)

\bibitem{EKM} Erd\H{o}s, P.L. - Kir\'aly, Z. - Mikl\'os, I.: On graphical degree sequences and realizations, {\sl Combinatorics, Probability and Computing} (2013), 1--22. {\tt doi:10.1017/S0963548313000096}

\bibitem{EMT09} Erd\H{o}s, P.L. - Mikl\'os, I. - Toroczkai, Z.: A simple Havel-Hakimi type algorithm to realize graphical degree sequences of directed graphs,  {\sl  Elec. J. Combinatorics} {\bf 17} (1) (2010),   R66 (10pp)

\bibitem{G57} Gale, D.: A theorem on flows in networks, {\sl Pacific J. Math.} {\bf 7}  (2)  (1957), 1073--1082.

\bibitem{green} Greenhill, C.:  A polynomial bound on the mixing time of a Markov chain for sampling regular directed graphs, {\sl Elec. J. Combinatorics} {\bf 18} (2011), \#P234.

\bibitem{H62} Hakimi, S.L.: On the realizability of a set of integers as degrees of the vertices of a simple graph. {\sl J. SIAM Appl. Math.} {\bf 10} (1962), 496--506.

\bibitem{H55} Havel, V.: A remark on the existence of finite graphs. (in Czech), {\sl \v{C}asopis P\v{e}st. Mat.} {\bf 80} (1955), 477--480.

\bibitem{JVV}  Jerrum, M. R. -  Valiant,  L. G. -  Vazirani, V. V.: Random generation of combinatorial structures from a uniform distribution, {\sl Theoret. Comput. Sci.}, {\bf 43} (2-3) (1986), 169--188.

\bibitem{JSV} Jerrum, M.R. - Sinclair, A. - Vigoda, E.: A Polynomial-Time Approximation Algorithm for the Permanent of a Matrix with Nonnegative Entries, {\sl Journal of the ACM} {\bf 51}(4) (2004), 671--697.

\bibitem{KTV97} Kannan, R. - Tetali, P. - Vempala, S.: Simple Markov-chain algorithms for generating bipartite graphs and tournaments, {\sl Rand. Struct. Alg.} {\sl 14} (4) (1999), 293--308.

\bibitem{KTEMS} Hyunju Kim - Toroczkai, Z. - Erd\H{o}s, P.L. - Mikl\'os, I. - Sz\'ekely, L.A.: Degree-based graph construction, {\sl J. Phys. A: Math. Theor.} {\bf 42} (2009) 392001 (10pp)

\bibitem{KW73} D.J.~Kleitman - D.L.~Wang: Algorithms for constructing graphs and digraphs with given valences and factors, {\sl  Discrete Math.} {\bf 6} (1973), 79--88.

\bibitem{lamar}  M.D.~LaMar: Directed 3-Cycle Anchored Digraphs And Their Application In The Uniform Sampling Of Realizations From A Fixed Degree Sequence, in {\sl ACM \& IEEE \& SCS Proc. of 2011 Winter Simulation Conference} (Eds. S. Jain, R.R. Creasey {\sl et. al.}) (2011), 1--12.

\bibitem{MES} Mikl\'os, I. - Erd\H{o}s, P.L. - Soukup, L.: Towards  random uniform sampling of bipartite graphs with given degree sequence,  {\sl Electronic J. Combinatorics} {\bf 20} (1) (2013), \#P16, 1--49.

\bibitem{NBW06} Newman, M.E.J. - Barab\'asi, A.L. - Watts, D.J.: {\it The Structure and Dynamics of Networks} (Princeton Studies in Complexity, Princeton UP) (2006), pp 624.

\bibitem{newman} Newman, M.E.J.: {\it Networks: An Introduction} Oxford University Press, March 2010, pp. 784.

\bibitem{pet} J.~Petersen: Die Theorie der regularen Graphen, {\sl Acta Math.} {\bf 15} (1891), 193--220.

\bibitem{R57} Ryser, H. J.: Combinatorial properties of matrices of zeros and ones, {\sl Canad. J. Math.} {\bf 9} (1957), 371--377.

\bibitem{S92} Sinclair, A.: Improved bounds for mixing rates of Markov chains and multicommodity flow, {\sl Combin. Probab. Comput.} {\bf 1} (1992), 351--370.

\bibitem{T52} W.T.~Tutte: The factors of graphs, {\sl Canad. J.  Math.} {\bf 4} (1952), 314--328.



\end{thebibliography}
\end{document}